\documentclass[12pt]{amsart}
\usepackage{amsmath}
\usepackage{amssymb}
\usepackage[mathcal]{eucal}
\usepackage[all]{xy}
\usepackage{latexsym}
\usepackage{amstext}
\usepackage{amsfonts}
\usepackage{amsthm}
\usepackage{amsopn}
\usepackage{amsbsy}
\usepackage{layout}
\usepackage{color}
\usepackage{graphicx}
\usepackage{bm}
\usepackage{yfonts}
\usepackage{pifont}
\usepackage[scr]{rsfso}
\usepackage{tikz}
\newtheorem{thm}{Theorem}[section]

\newtheorem{cor}[thm]{Corollary}
\newtheorem{pro}[thm]{Proposition}

\newtheorem{prob}[thm]{Problem}

\newtheorem{df}[thm]{Definition}

\newtheorem{claim}{Claim}

\def\N{\mathbb N}
\def\R{\mathbb R}

\def\supp{{\mathrm {supp}}\,}

\begin{document}

\title[Linear and uniformly continuous surjections]{Linear and uniformly continuous surjections between $C_p$-spaces over metrizable spaces}
\author[A. Eysen, A. Leiderman and V. Valov]
{Al\.{i} Emre Eysen, Arkady Leiderman and Vesko Valov}
\address{Department of Mathematics, Faculty of Science, Trakya University, Edirne, Turkey}
\email{aemreeysen@trakya.edu.tr}
\thanks{The first author was partially supported by TUBITAK-2219}

\address{Department of Mathematics, Ben-Gurion University of the Negev,
Beer-Sheva, Israel}
\email{arkady@math.bgu.ac.il}

\address{Department of Computer Science and Mathematics, Nipissing University,
100 College Drive, P.O. Box 5002, North Bay, ON, P1B 8L7, Canada}
\email{veskov@nipissingu.ca}
\thanks{The third author was partially supported by NSERC Grant 261914-19}



\begin{abstract}
For any Tychonoff space $X$ let $D(X)$ be either the set $C(X)$ of all continuous functions on $X$ or the set $C^*(X)$ of all bounded continuous functions on $X$. When $D(X)$ is endowed with the point convergence topology, we write $D_p(X)$. 

 Let $T: D_{p}(X) \to D_{p}(Y)$ be a continuous linear surjection, where $X$ is a metrizable space and $Y$ is perfectly normal. 
We show that if $X$ has some dimensional-like property $\mathcal P$, then so does $Y$. For example, $\mathcal P$ could be one of the following properties: zero-dimensionality, countable-dimensionality or strong countable-dimensionality. This result remains true if $T$ is a uniformly continuous and inversely bounded surjection.

Also, we consider other properties $\mathcal P$: of being a scattered, or a strongly $\sigma$-scattered space, or being a $\Delta_1$-space, see \cite{KKL}.
Our results strengthen and extend several results from \cite{b3}, \cite{gkm} and \cite{KKL}.
\end{abstract}

\maketitle




\section{Introduction}\label{intro}
For a Tychonoff space $X$,
by $C(X)$ we denote the linear space of all continuous real-valued functions on $X$.
$C^*(X)$ is a subspace of $C(X)$ consisting of the bounded functions.
We write $C_p(X)$ (resp., $C_p^*(X)$) if $C(X)$ (resp., $C^*(X)$) is endowed with the pointwise convergence topology.
The questions concerning linear or uniform homeomorphisms of $C_p$-spaces have been intensively studied by many authors.
More information can be found in 
\cite{ar1}, \cite{mar}, \cite{vanMill}, \cite{tk}, \cite{tk2}.

Throughout the paper by dimension we mean the {\em covering dimension $\dim$}.
Recall that for a Tychonoff space $X$ and an integer $n\geq 0$, $\dim X \leq n$ if every finite functionally open cover of the space $X$
 has a finite functionally open refinement of order $\leq n$, see \cite{en}.

After the striking results of Pestov \cite{p} and Gul'ko \cite{gu} that $\dim X=\dim Y$ for any Tychonoff spaces $X$ and $Y$ 
provided $C_p(X)$ and $C_p(Y)$ are linearly homeomorphic or uniformly homeomorphic, Arhangel'skii posed a problem
whether $\dim Y\leq\dim X$ if there is continuous linear surjection from $C_p(X)$ onto $C_p(Y)$, see \cite{ar}.
This question was answered negatively by Leiderman-Levin-Pestov \cite{llp} and Leiderman-Morris-Pestov \cite{lmp}. 
For every finite-dimensional metrizable compact space $Y$ there exists a continuous linear surjection $T: C_p([0,1])\to C_p(Y)$ \cite{lmp}. Later, Levin \cite{Levin}
 showed that one can construct such a surjection which additionally is an open mapping. 

However, it turned out that the zero-dimensional case is an exception.  It was shown in \cite{llp} that if there is a linear continuous surjection $T: C_p(X)\to C_p(Y)$ for compact metrizable spaces $X$ and $Y$,
 then $\dim X=0$ implies that $\dim Y=0$. The last result was extended for arbitrary compact spaces by Kawamura-Leiderman \cite{KawL}
 who also raised the question whether the same statement is true without the assumption of compactness of $X$ and $Y$.
 Recently, this difficult question was answered positively in \cite{ev}.

Everywhere below, by $D(X)$ we denote either $C^*(X)$ or $C(X)$, and $D_p(X)$ stays for $D(X)$ endowed with the pointwise convergence topology.  
In the present paper we mainly focus on linear or uniformly continuous surjections $T: D_p(X)\to D_p(Y)$, where $X$ is a metrizable space, $Y$ is either metrizable or perfectly normal and $T$ satisfies some additional condition. 
Moreover, almost all results are true if we consider any one of the four possible cases: $D(X)$ is either $C(X)$ or $C^*(X)$ and $D(Y)$ is either $C(Y)$ or $C^*(Y)$. So, everywhere below, if not said otherwise, we assume that all four cases are considered.

\begin{itemize}
\item A map $T: D_p(X)\to D_p(Y)$ is called {\em uniformly continuous} if for every neighborhood $U$ of the zero function in $D_p(Y)$ there is a neighborhood $V$ of the zero function in $D_p(X)$ such that $f,g\in D_p(X)$ and $f-g\in V$ implies $T(f)-T(g)\in U$.
\item For every bounded function $f\in C(X)$ by $||f||$  we denote its {\em supremum}-norm.
A map $T: D(X)\to D(Y)$ is called {\em $c$-good} (see \cite{gf}, \cite{gkm}) if for every $g\in C^*(Y)$ there exists  $f\in C^*(X)$ 
such that $T(f)=g$ and $||f||\leq c ||g||$.
\end{itemize}

$\N$ denotes the set of natural numbers $\{1, 2, \dots\}$.
We say that a sequence $\{g_n: n\in\N\}\subset C^*(Y)$ is {\em norm bounded} if there is $M > 0$ such that $||g_n||\leq M$ for all $n\in \N$.

\begin{df}\label{df1}
A map $T: D(X)\to D(Y)$ is called inversely bounded if for every norm bounded sequence
 $\{g_n\}\subset C^*(Y)$ there is a norm bounded sequence $\{f_n\}\subset C^*(X)$ with $T(f_n)=g_n$ for each $n\in\N$. 
\end{df}

 Evidently, every linear continuous map between $D_p(X)$ and $D_p(Y)$ is uniformly continuous
and every $c$-good map is inversely bounded.  Also, every linear continuous surjection
$T: C_p^*(X)\to C_p^*(Y)$, where $X$ and $Y$ are arbitrary Tychonoff spaces, is inversely bounded, see \cite[Proposition 3.3]{ev}.
\newpage
Recall that a normal topological space $X$ is {\em countable-dimensional} ({\em strongly countable-dimensional})
if $X$ can be represented as a countable union of normal finite-dimensional subspaces (resp., closed finite-dimensional subspaces).
Note that there are countable-dimensional compact metrizable spaces which are not strongly countable-dimensional, see \cite{en}.

Marciszewski \cite[Corollary 2.7]{mar} observed that, by modifying the Gulko's arguments \cite{gu} (c.f. \cite{mar2}, \cite{mp}), one can show
the following theorem:
\begin{thm} 
Let $\mathcal P$ be the property of metrizable spaces such that:
\begin{itemize}
\item [(i)] If $X\in \mathcal P$ and $Y$ is a subset of $X$ then $Y \in \mathcal P$,
\item [(ii)] If $X$ is a metrizable space which is a countable union of closed subsets $X_n\in \mathcal P$  then $X\in \mathcal P$.
\end{itemize}
Then, for metrizable spaces $X$ and $Y$ such that $C_p(X)$ and $C_p(Y)$ are uniformly homeomorphic,
$X\in \mathcal P$ if and only if $Y\in \mathcal P$.
\end{thm}

It is known that the covering dimension $\dim X \leq n$ satisfies the above conditions  (i) and (ii).
Much less is known about the case when $T: C_{p}(X) \to C_{p}(Y)$ is supposed to be only uniformly continuous and surjective. 
The following open problem has been posed in \cite[Question 4.1]{gkm}.

\begin{prob}\label{problem1}
Let $X$ be a compact metrizable strongly countable-dimensional {\em[}zero-dimensional{\em]} space.
Suppose that there exists a uniformly continuous surjection $T: C_{p}(X) \to C_{p}(Y)$.
Is $Y$ necessarily strongly countable-dimensional {\em[}zero-dimensional{\em]}?
\end{prob}

The authors of \cite{gkm} established that the answer to Problem \ref{problem1} is affirmative provided that
the uniformly continuous surjection $T$ is $c$-good for some $c > 0$. 
Later, in \cite{ev} the same was proved for $\sigma$-compact metrizable spaces.
In our paper  we strengthen this result by proving this remains true for all metrizable spaces and all inversely bounded 
uniformly continuous surjections.
In fact we develop a general scheme for the proof as follows. 

We consider the properties $\mathcal P$ of normal spaces such that:
\begin{itemize}
\item [(a)] if $X\in\mathcal P$ and $F\subset X$ is closed, then $F\in\mathcal P$;
\item [(b)] $\mathcal P$ is closed under finite products;
\item [(c)] if $X$ is a countable union of closed subsets each having the property $\mathcal P$, then $X\in\mathcal P$;
\item [(d)] if $f:X\to Y$ is a closed map with finite fibers, where $Y$ is a metrizable space with $Y\in\mathcal P$, then $X\in\mathcal P$.
\end{itemize}

From the classical results of dimension theory (see \cite{en}) it follows that {\em zero-dimensionality}, {\em countable-dimensionality} and {\em strongly countable-dimensionality} satisfy conditions $(a)-(d)$ above.

Now  we formulate one of the main results of our paper. 
\begin{thm}\label{theorem-main} Let $X$ be a metrizable space and $Y$ be perfectly normal.
Suppose that $T: D_p(X)\to D_p(Y)$ is a uniformly continuous inversely bounded surjection.
For any topological property $\mathcal P$ satisfying conditions $(a)-(d)$ above, if $X\in\mathcal P$ then $Y\in\mathcal P$.
\end{thm} 

\begin{cor}\label{cor-main} Let $X$, $Y$ and $T: D_p(X)\to D_p(Y)$ satisfy the hypotheses of Theorem $\ref{theorem-main}$.
\begin{itemize}
\item [(i)] If $X$ is either countable-dimensional or strongly countable-dimensional, then so is $Y$. 
\item [(ii)] If $X$ is zero-dimensional, then so is $Y$. 
\end{itemize}
\end{cor}

Note that item (ii) was established in \cite[Theorem 1.1]{ev} for arbitrary Tychonoff spaces $X, Y$ and $c$-good surjections $T$.
However, we don't know whether there exists a uniformly continuous inversely bounded map which is not $c$-good for some $c>0$.

A linear continuous version of Theorem \ref{theorem-main} is also true (for metrizable compact spaces it was implicitly established in \cite{llp}).

\begin{thm}\label{linear} Let $X$ be a metrizable space and $Y$ be a perfectly normal space.
Suppose that $T: D_p(X)\to D_p(Y)$ is a linear continuous surjection.
For any topological property $\mathcal P$ satisfying conditions $(a)-(d)$ above, if $X\in\mathcal P$ then $Y\in\mathcal P$.
\end{thm}

In the last part of our paper we show that for any two 
metrizable spaces $X$ and $Y$ such that there is a linear continuous surjection $T: C_p(X)\to C_p(Y)$ or $T: C_p^*(X)\to C_p^*(Y)$,
and $X$ is scattered, then so is $Y$ (Theorem \ref{linear2}).
Also, we apply Theorem \ref{theorem-main} to the property $\mathcal P$
of being a strongly $\sigma$-scattered space and to the property $\mathcal P$ of being a $\Delta_1$-space. 
All necessary definitions are given in Section \ref{section4}.

\section{Proof of Theorem \ref{theorem-main}}
Our proof is based on the idea of support introduced by Gul'ko \cite{gu}, see also \cite{mp}, where this technique is well described. 
For every $y\in Y$ there is a map $\alpha_y:D_p(X)\to\mathbb R$, $\alpha_y(f)=T(f)(y)$. Since $T$ is uniformly continuous, so is each $\alpha_y$. 
For every $y\in Y$ we consider the family $\mathcal A(y)$ of all finite sets $K\subset X$ such that $a(y,K)<\infty$, where 
$$a(y,K)=\sup\{|\alpha_y(f)-\alpha_y(g)|:f, g\in D(X), |f(x)-g(x)|<1{~}\forall x\in K\}.$$
Note that $a(y,\varnothing)=\infty$ since $T$ is surjective.
For every  $p, q \in \N$ we define
$$Y(p,q)=\{y\in Y:\exists K\in\mathcal A(y){~}\hbox{with}{~} a(y,K)\leq p{~}\hbox{and}{~}|K|\leq q\}$$
and $M(p)=\bigcup\{M(p,q): q\in \N \}$, where
$$M(p,q)=\{y\in Y(p,q):a(y,K)>2p{~} \forall K\subset X{~}\hbox{with}{~}|K|\leq q-1\}.$$
Gul'ko's methodology in \cite{gu} (see also \cite{mp}) was developed for metrizable spaces, but the extension of Gul'ko arguments from \cite{k} and \cite{ev} shows that for every Tychonoff spaces $X$ and $Y$ all conditions $(1)- (8)$ below are valid:

\begin{itemize}
\item[(1)] $\mathcal A(y)$ is non-empty and it is closed under finite intersections. Moreover, $a(y,K_1\cap K_2)\leq a(y,K_1)+a(y,K_2)$ for all $K_1,K_2\in\mathcal A(y)$;
\item[(2)] $Y(p,q)$ is closed in $Y$ for all $p,q\in\N$;
\item[(3)] $y\in M(p)$ provided $p\geq a(y)$;
\item[(4)] $M(p,1)=Y(p,1)$ and $M(p,q)=Y(p,q)\setminus Y(2p,q-1)$ for $q\geq 2$;
\item[(5)] $Y=\bigcup\{M(p,q): p, q\in\N\}$;
\item[(6)] $M(p,q_1)\cap M(p,q_2)=\varnothing$ for $q_1\neq q_2$;
\item[(7)] For every $y\in M(p,q)$ there is a unique finite $K_p(y)\subset X$ with $|K_p(y)|=q$ and $a(y,K_p(y))\leq p$;
\item[(8)] The map $\varphi_{pq}:M(p,q)\to [X]^q$, $\varphi_{pq}(y)=K_p(y)$, is continuous, where $[X]^q$ denotes the set of all $q$-point subsets of $X$ with the Vietoris topology. 
\end{itemize}
Moreover, if $X$ is metrizable and $Y$ is normal, the map $\varphi_{pq}$ satisfies the following additional condition:
\begin{claim}\label{claim1}
If $M\subset M(p,q)$ is closed in $Y$ for some $p,q$, then the map $\varphi_{pq}\restriction_M: M\to [X]^q$ is closed and each fiber of 
$\varphi_{pq}\restriction_M$ is countably compact. 
\end{claim}
Because $[X]^q$ is a metrizable space, it suffices to show that if $\{y_n\}$ is a sequence in $M$ such that $\varphi_{pq}(y_n)$ converges to some $K\in [X]^q$, then $\{y_n\}$ has an accumulation point in $M$. That statement was established in the proof of condition $(9)$ from \cite{mp} in case both $X$ and $Y$ are metrizable, but the same proof works when $X$ is metrizable and $Y$ is normal.

Since each $Y(p,q)$ is a closed subset of $Y$, it follows from $(4)$ that each $M(p,q)$ is a countable union of closed subsets $\{F_n(p,q): n\in\N\}$ of $Y$. 
So, by $(5)$, $Y=\bigcup\{F_n(p,q): n, p, q\in\N\}$. According to Claim 1, all maps $\varphi_{pq}^n=\varphi_{pq}\restriction_{F_n(p,q)}: F_n(p,q)\to [X]^q$ are closed and have countably compact fibers.
\begin{claim}\label{claim2}
The fibers of $\varphi_{pq}^n: F_n(p,q)\to [X]^q$ are finite. 
\end{claim} 
We follow the arguments from the proof of \cite[Theorem 4.2]{gkm}.  
Fix $z\in F_n(p,q)$ for some $n, p, q\in\N$ and let $A(z)=\{y\in F_n(p,q):K_p(y)=K_p(z)\}$.  Suppose that $A(z)$ is infinite, so it contains a sequence $S=\{y_m\}$ of distinct points. Because, by Claim 1, $A(z)$ is countably compact, there are two possibilities: either $\{y_m\}$ is closed and discrete or it contains an accumulation point in $A(z)$. Therefore, passing to a subsequence, we may assume that 
for every $y_m$ there exist a neighborhood $U_m$ in $Y$ and a function $g_m: Y\to [0,2p]$ such that: $U_m\cap S=\{y_m\}$,
$g_m(y_m)=2p$ and $g_m(y)=0$ for all $y\not\in U_m$.
 Since $T$ is inversely bounded, there is a norm bounded sequence $\{f_m\}\in C^*(X)$ with $T(f_m)= g_m$. Let $r > 0$ be such that 
$||f_m||\leq r$, $m\in\N$. So, the sequence $\{f_m\}$ is contained in the compact set $[-r, r]^X$. Hence,  $\{f_m\}$ has an accumulation point in  $[-r, r]^X$. This implies the existence of $i\neq j$ such that $|f_i(x)-f_j(x)| < 1$ for all $x\in K_p(z)$.
 Consequently, since $K_p(y_j)=K_p(z)$, $|\alpha_{y_j}(f_j)-\alpha_{y_j}(f_i)|\leq p$. On the other hand, $\alpha_{y_j}(f_j)=T(f_j)(y_j)=g_j(y_j)=2p$ and
$\alpha_{y_j}(f_i)=T(f_i)(y_j)=g_i(y_j)=0$, so $|\alpha_{y_j}(f_j)-\alpha_{y_j}(f_i)|=2p$, a contradiction.

Now we can complete the proof of Theorem \ref{theorem-main}. Suppose that $X$ has a property $\mathcal P$ satisfying conditions $(a)-(d)$. Then so does $X^q$ for each $q$. 
The space $[X]^q$ is homeomorphic to the set $W_q=\{(x_1, x_2, \dots, x_q)\in X^q: x_i\neq x_j{~}\mbox{for}{~}i\neq j\}$ which is open in  $X^q$.
 So, $[X]^q\in\mathcal P$ as a countable union of closed subsets of $X^q$. According to Claim 1, $\varphi_{pq}^n(F_n(p,q))$ is closed in $[X]^q$. Hence,
$\varphi_{pq}^n(F_n(p,q))$ has the property $\mathcal P$. Finally, since the map $\varphi_{pq}^n: F_n(p,q)\to\varphi_{pq}^n(F_n(p,q))$ is perfect and has finite fibers, we obtain  $F_n(p,q)\in\mathcal P$. 
Therefore, by condition $(c)$, $Y=\bigcup\{F_n(p,q): n, p, q\in\N\}$ also has the property $\mathcal P$. 
\hfill $\Box$

\section{Proof of Theorem \ref{linear}}

Suppose that $X$ and $Y$ are Tychonoff spaces and $T: D_p(X)\to D_p(Y)$ is a continuous linear surjection.
Every $y\in Y$ generates a linear continuous map $l_y: D_p(X)\to\mathbb R$ defined by $l_y(f)=T(f)(y)$. 
It is well known, see for example \cite{ar1} or \cite{bd}, that for every $l_y$ there exist a finite set
 $\supp(l_y)=\{x_1(y), x_2(y), \dots, x_k(y)\}$ in $X$ and real numbers $\lambda_i(y)$, $i = 1, 2, \dots, k$,
such that for all $f\in D_p(X)$ we have $l_y(f)=\sum_{i=1}^k\lambda_i(y)f(x_i(y))$. 
 Here we recall some properties of the supports $\supp(l_y)$, see \cite{bd} and \cite[Section 6.8]{vanMill}.

\begin{itemize}
\item[(P1)] If $f\restriction_{\supp(l_y)})=g\restriction_{\supp(l_y)}$ for some $f, g\in D(X)$, then $l_y(f)=l_y(g)$;
\item[(P2)] If $\supp(l_{y_0})\cap U\neq\varnothing$ for some open $U\subset X$ and $y_0\in Y$, 
then $y_0$ has a neighborhood $V\subset Y$ such that $\supp(l_{y})\cap U\neq\varnothing$ for every $y\in V$;
\item[(P3)] Every set $Y_k=\{y\in Y: |\supp(l_y)|\leq k\}$ is closed in $Y$; 
\end{itemize}

A subset $A$ of a space $X$ is {\em bounded} if $f(A)$ is a bounded set in $\mathbb R$ for every $f\in C(X)$.
The following property is valid only in the case when $T$ is a continuous linear surjection between $C_p$-spaces, not for $C_p^*$-spaces.
\begin{itemize}
\item[(P4)] 
Suppose that $X$ and $Y$ are Tychonoff spaces and $T: C_p(X)\to C_p(Y)$ is a continuous linear surjection.
If $A\subset X$ is bounded, then so is the set $\{y\in Y: \supp(l_y)\subset A\}$.
\end{itemize}

In the case of $C_p^*$-spaces, we will use the following property, see \cite[Lemma 1.4.6]{bd}:

\begin{itemize}
\item[(P5)] 
Suppose that $X$ and $Y$ are metrizable spaces and $T:C_p^*(X)\to C_p^*(Y)$ is a continuous linear surjection.
 If $A\subset X$ is compact, then the set $\{y\in Y: \supp(l_y)\subset A\}$ is also compact. 
\end{itemize}

We consider the sets $M_1=Y_1$ and $M_k=Y_k\setminus Y_{k-1}$ for $k\geq 2$.
 Let $S_k: M_k\to [X]^k$ be the map defined by $S_k(y)=\supp(l_y)$. It follows from $(P2)$ that $S_k$ is continuous.
\begin{claim}\label{claim3}
Let $F\subset M_k$ be closed in $Y$ for some $k$. Then the map $S_k\restriction_F: F\to [X]^k$ is closed.
\end{claim}
Since $X$ is a metrizable space, it suffices to show that if $\{y_n\}$ is a sequence in $F$ and $S_k(y_n)$ converges to some $K\in[X]^k$, then
$\{y_n\}$ has an accumulation  point in $F$. Striving for a contradiction, suppose that there is a sequence $\{y_n\}$ in $F$ such that the set 
$Z=\{y_n: n\in\N\}$ is closed and discrete in $Y$. Let $K=\{x_1, x_2, \dots, x_k\}$ and $S_k(y_n)=\{x_1(y_n), x_2(y_n), \dots, x_k(y_n)\}$ for all $n$.
 Since $S_k(y_n)$ converges to $K$ in $[X]^k$, each of the sequences $\{x_i(y_n)\}_{n\in\N}$, $i=1, 2, \dots, k$, converges in $X$ to $x_i$.  Therefore,  
$A=\bigcup_{i=1}^k\{x_i\}\cup\{x_i(y_n)\}_{n\in\N}$ is a compact subset of $X$. Because $X$ is a metrizable space, according to Dugundji Extension Theorem \cite{d}, (see also \cite{vanMill}), there is a continuous linear map $\Theta: C_p(A)\to C_p^*(X)$ such that $\Theta(g)\restriction_{A} = g$ for all $g\in C(A)$. Thus, the linear map
$\varphi: C_p(A)\to D_p(Z)$, $\varphi(g)=T(\Theta(g))\restriction_{Z}$, is continuous, where $D(Z)=C(Z)$ if $D(Y)=C(Y)$ and $D(Z)=C^*(Z)$ if $D(Y)=C^*(Y)$. 
Moreover, $\varphi$ is surjective. Indeed, take $h\in D(Z)$ and its continuous extension $\overline h\in D(Y)$. Then $T(f)=\overline h$ for some $f\in D(X)$
 and the functions $f$ and $g=\Theta(f\restriction_{A})$ have the same restrictions on $A$. Hence, by  $(P1)$, $l_y(f)=l_y(g)$ for all $y\in Z$.
 Thus, $\varphi(f\restriction_{A}) = h$. If $D(Z)=C(Z)$, then by $(P4)$, the set $Z$ is bounded. If $D(Z)=C^*(Z)$, according to $(P5)$, $Z$ is also bounded. Therefore, in both possible cases we have a contradiction.

Since each $Y_k$ is closed in $Y$ and $Y$ is perfectly normal, $M_{k}$ is the union of a countably many closed subsets $M_{kn}$ of $Y$.
 Let $S_{kn}=S_k\restriction_{M_{kn}}$. According to Claim \ref{claim3}, each $S_{kn}$ is a closed map.
\begin{claim}\label{claim4}
The fibers of each map $S_{kn}: M_{kn}\to [X]^k$ are finite.
\end{claim}
Indeed, let $z\in M_{kn}$, $S_{kn}(z)=\{x_1, x_2, \dots, x_k\}$ and $A(z)=\{y\in M_{kn}:S_{kn}(y)=S_{kn}(z)\}$. Since $\supp(l_y)=S_{kn}(z)$ for all $y\in A(z)$, as in the proof of Claim 3, there is a continuous linear surjection  $\phi:C_p(S_{kn}(z))\to D_p(A(z))$. Because $C_p(S_{kn}(z))$ is finite-dimensional, by linearity, so is $D_p(A(z))$.
Thus, $A(z)$ is finite.

Finally, as in the last paragraph from the proof of Theorem \ref{theorem-main}, we can show that $Y$ has the property $\mathcal P$. 
\hfill $\Box$ 

\section{Scattered-like properties $\mathcal P$}\label{section4}
To begin with, we recall several notions and facts (probably well-known) which will be discussed in this section.
 A space $X$ is said to be {\em scattered} if every nonempty
subset $A$ of $X$ has an isolated point in $A$.
A Tychonoff scattered space need not to be zero-dimensional \cite{Solomon},
while every metrizable scattered space is completely metrizable (see, for instance, \cite{Michael}) and zero-dimensional.

A space $X$ is said to be (strongly) {\em $\sigma$-scattered} if $X$ can be represented as a countable union of (closed) scattered subspaces,
and $X$ is called (strongly) {\em $\sigma$-discrete} if  $X$ can be represented as a countable union of (closed) discrete subspaces.
By the classical result of Stone \cite{Stone}, all these four properties are equivalent in the class of metrizable spaces (for a more modern treatment of this result see \cite{Plewik}). Hence, every metrizable $\sigma$-scattered space must be zero-dimensional.

Recall the following results of Baars:

\begin{thm}\label{Baars1} \cite{b1} Let $X$ and $Y$ both be first countable paracompact spaces.
Suppose that $T: C_p(X) \to C_p(Y)$ is a linear homeomorphism. Then $X$ is scattered if and only if $Y$ is scattered.
\end{thm}

\begin{thm}\label{Baars2} \cite{b2}, \cite{b3}
Let $X$ and $Y$ be metrizable spaces.
Suppose that $T: C_p^*(X) \to C_p^*(Y)$ is a linear homeomorphism. Then $X$ is scattered if and only if $Y$ is scattered.
\end{thm}

It is an open problem whether Theorem \ref{Baars2} remains true if both $X$ and $Y$ are assumed to be first countable paracompact spaces
(\cite[Question 4.8]{b3}), despite of the following structural result which apparently is due to Telg\' arsky \cite[Theorem 8]{Tel}: every scattered first countable paracompact space is metrizable 
and strongly $\sigma$-discrete.  

Below we strengthen both Theorems \ref{Baars1} and \ref{Baars2} in case that $X$ and $Y$ are metrizable spaces,
assuming only that $T$ is a linear continuous surjection.

We will use several well-known facts about the linear topological spaces, which are dual to $D_p(X)$.
In essence, we have already described some properties of the dual to $D_p(X)$ in the proof of Theorem \ref{linear}.
We repeat that for any Tychonoff space $X$ the dual space of $D_p(X)$ algebraically can be identified with a linear space of formal linear combinations $L(X)$,
where $X$ is a Hamel basis in $L(X)$. For each natural $n\in \N$
denote by $M_n(X)$ the subspace of $L(X)$ formed by all linear combinations of the reduced length precisely $n$.
Let $\tau_p$ be the topology on $L(X)$ when $L(X)$ is considered as a weak topological dual to $C_p(X)$,
and let $\tau_b$ be the topology on $L(X)$ when $L(X)$ is considered as a weak topological dual to $C_p^*(X)$, respectively.
In general, $\tau_b$ does not coincide with $\tau_p$ on the whole linear space $L(X)$ \cite{BMT}. 
However, the analysis of the proof of \cite[Proposition 0.5.17]{ar1} easily shows that the topologies $\tau_b$ and $\tau_p$ restricted to
subsets $M_n(X)$ do coincide. (One need to do some cosmetic changes which are based on the following trivial remark: for any point $x\in X$ and open $U \subset X$ containing $x$ there is a {\em bounded} continuous function $f$ on $X$ such that
$f(x)=1, f\restriction_{X\setminus U} = 0$). Hence, from \cite[Proposition 0.5.17]{ar1} (see also \cite[Proposition 2.1]{KawL}) we can deduce the following result.

\begin{pro}\label{ar1}
$(M_n(X), \tau_p) = (M_n(X), \tau_b)$ is homeomorphic to a subspace of the Tychonoff product $(\R^*)^n\times X^n$, where $\R^*= \R\setminus\{0\}$. 
\end{pro}

Now we have a result for all Tychonoff spaces (everywhere below, except for Theorem \ref{linear2}, we suppose that $T: D_p(X)\to D_p(Y)$ is a surjection such that all possible four cases are considered).

\begin{pro}\label{linear3} Let $X$ and $Y$ be Tychonoff spaces.
Suppose that $T: D_p(X)\to D_p(Y)$ is a linear continuous surjection. 
If $X$ is $\sigma$-scattered ($\sigma$-discrete), then $Y$ also is $\sigma$-scattered ($\sigma$-discrete, respectively). 
\end{pro}

\begin{proof}
The adjoint mapping $T^*$ isomorphically embeds the dual space of $D_p(Y)$, i.e. $(L(Y), \tau_p)$ or $(L(Y), \tau_b)$ into the dual space of $D_p(X)$, i.e. $(L(X), \tau_p)$ or
 $(L(X), \tau_b)$.
In all cases, by Proposition \ref{ar1} the space
$Y$ can be represented as a countable union of
subspaces $Y_i, i \in \N$, such that each $Y_i$ is homeomorphic to a subspace of $(\R^*)^n\times X^n$ for some $n=n(i)$.

Consider the projection $p_i$ of each of the above pieces $Y_i \subset (\R^*)^n\times X^n$ to the second factor $X^n$.
The following property of projections $p_i$ can be recovered from the proof of \cite[Proposition 2.1]{KawL}.
Indeed, \cite[Proposition 2.1]{KawL} has been formulated and proved assuming that $X$ and $Y$ are metrizable compact spaces,
however, the proof of the following claim which is a part of the proof of \cite[Proposition 2.1]{KawL} is valid for any Tychonoff spaces $X$ and $Y$.

\vspace{0.08in}
{\bf Claim.} Every projection $p_i: Y_i \longrightarrow X^n$ is a finite-to-one mapping.
\vspace{0.08in}
\newline
Evidently, $X^n$ is $\sigma$-scattered (resp., $\sigma$-discrete) provided so is $X$. Since $p_i$ is continuous and its fibers are finite, for every $Z\subset X^n$ and every isolated point $z$ in $Z$ the fiber 
$p_i^{-1}(z)$ consists of isolated points in $p_i^{-1}(Z)$. We conclude, if $X$ is $\sigma$-scattered ($\sigma$-discrete), then each $Y_i, i \in \N$ is $\sigma$-scattered ($\sigma$-discrete, respectively),
and then $Y$ is $\sigma$-scattered ($\sigma$-discrete, respectively).
\end{proof}

\begin{thm}\label{linear2} Let $X$ and $Y$ be metrizable spaces.
Suppose that $T: D_p(X)\to D_p(Y)$ is a linear continuous surjection such that either $D(X)=C(X)$ and $D(Y)=C(Y)$ or $D(X)=C^*(X)$ and $D(Y)=C^*(Y)$. 
If $X$ is scattered, then so is $Y$. 
\end{thm}

\begin{proof} In the case $T: C_p(X)\to C_p(Y)$ the statement has been formulated and proved earlier \cite[Proposition 3.9]{KL2}. 
Here we provide a simpler and unified proof for both options.
By Proposition \ref{linear3} $Y$ is $\sigma$-scattered. Since $Y$ is metrizable, $Y$ is strongly $\sigma$-discrete, by the aforementioned result of Stone \cite{Stone}.
From another hand, every metrizable and scattered space is completely metrizable.
Therefore, $X$ is completely metrizable and then $Y$ also is completely metrizable, by the main result of \cite{Pelant}.
Finally, by the Baire category theorem every \v{C}ech-complete strongly $\sigma$-discrete space is scattered.
\end{proof}

We don't know whether analogues of Theorem \ref{linear2} and Proposition \ref{linear3}
are valid under a weaker assumption: $T: D_p(X)\to D_p(Y)$ is a uniformly continuous surjection.
This is because the proof of Theorem \ref{linear2} relies on the result of Baars-de Groot-Pelant \cite{Pelant} (that completeness is preserved
by continuous linear surjections $T:C_p(X)\to C_p(Y)$ or $T:C_p^*(X)\to C_p^*(Y)$), while the following major question posed by Marciszewski and Pelant is still open.

\begin{prob}\label{problem2} {\em (See \cite[2.18. Problem]{mar})}
Let $X$ and $Y$ be (separable) metrizable spaces and let $T: D_p(X) \to D_p(Y)$ be a 
a uniformly continuous surjection (uniform homeomorphism). Let $X$ be completely metrizable.
Is $Y$ also completely metrizable?
\end{prob}

Moreover, the next problem is also open:

\begin{prob}\label{problem3}
Let $X$ and $Y$ be (separable) metrizable spaces and let $T: D_p(X) \to D_p(Y)$ be 
an inversely bounded uniformly continuous surjection. Let $X$ be completely metrizable.
Is $Y$ also completely metrizable?
\end{prob}

We obtain a $\sigma$-scattered analogue of Theorem \ref{linear2} for inversely bounded uniformly continuous surjections.

\begin{thm}\label{uniform_scattered} Let $X$ and $Y$ be metrizable spaces.
Suppose that $T: D_p(X)\to D_p(Y)$ is an inversely bounded uniformly continuous surjection. 
If $X$ is strongly $\sigma$-scattered, then $Y$ also is strongly $\sigma$-scattered.
\end{thm}
\begin{proof}
Any product of finitely many scattered (resp., strongly $\sigma$-scattered)  spaces is scattered (resp.,
strongly $\sigma$-scattered). Evidently, any closed subset of a strongly 
$\sigma$-scattered space is strongly $\sigma$-scattered. It is also true that the preimage of a strongly $\sigma$-scattered
space under a continuous map with finite fibers is strongly $\sigma$-scattered. Hence, applying Theorem \ref{theorem-main} we complete the proof.
\end{proof}

The last class of topological spaces that we consider in this paper is the class of $\Delta_1$-spaces.
A topological space $X$ is called a {\em $\Delta_1$-space} if any disjoint sequence $\{A_n : n\in\N\}$ of
countable subsets of $X$ has a point-finite open expansion, i.e. there exists a point-finite sequence
 $\{U_n : n\in\N\}$ of open subsets of $X$ such that $A_n \subseteq U_n$ for each $n\in\N$.
Equivalently, $X$ is a $\Delta_1$-space if any countable sequence of distinct points in $X$ has a point-finite open expansion \cite{KKL}.

The class of Tychonoff $\Delta_1$-spaces is tightly connected to certain properties of $C_p(X)$. 
Another motivation for studying the $\Delta_1$-spaces is provided by the fact
that it extends the classical notion of the $\lambda$-sets of reals. Recall that $X \subset \R$ is called
a $\lambda$-set if every countable $A \subset X$ is a $G_{\delta}$-subset of $X$ and, more generally,
a topological space $X$ is a $\lambda$-space if every countable subset $A \subset X$ is a $G_{\delta}$-subset.
 The study of $\lambda$-sets dates back to 1933 when Kuratowski proved in ZFC that there exist uncountable $\lambda$-sets.
According to \cite[Theorem 2.19]{KKL}, a metrizable space is a $\Delta_1$-space if and only if it is a $\lambda$-space.
If $X$ is a \v{C}ech-complete (in particular, if $X$ is a compact or a completely metrizable) space
then $X$ is a $\Delta_1$-space if and only if $X$ is scattered, see \cite[Corollary 2.16]{KKL}.

It was shown in \cite[Theorem 3.16]{KKL} that if $X$ and $Y$ are Tychonoff spaces and 
there is a linear continuous surjection $T: C_p(X) \to C_p(Y)$, 
 then $Y\in\Delta_1$ provided $X\in\Delta_1$. A slight modification of the proof shows that the same is true when
$T: C_p^*(X) \to C_p^*(Y)$.

\begin{thm}\label{uniform_delta}
Let $X$ and $Y$ be metrizable spaces.
Suppose that $T: D_p(X)\to D_p(Y)$ is an inversely bounded uniformly continuous surjection. 
If $X$ is a $\Delta_1$-space then $Y$ also is a $\Delta_1$-space.
\end{thm}
\begin{proof}
Obviously, the class $\Delta_1$ is hereditary with respect to any subspace. Moreover, the class $\Delta_1$ satisfies conditions $(b)$ and $(c)$, see Theorem 3.14 and Theorem 3.9 from \cite{KKL}. Further, it is easily seen that if we have a continuous mapping $\varphi: X\to Z$ with finite fibers and $Z\in\Delta_1$, then $X\in\Delta_1$. 
Therefore, we can apply Theorem \ref{theorem-main} to conclude that $X\in\Delta_1$ implies $Y\in\Delta_1$.
\end{proof}

\textbf{Acknowledgements.}
The authors are grateful to the referee for careful reading of the paper and valuable 
suggestions and comments.


\begin{thebibliography}{00}

\bibitem{ar}
A.~Arkhangel'skii, \textit{Problems in $C_p$-theory}, in J. van Mill and G. M. Reed eds., Open Problems in Topology, North-Holland (1990), 601--615.

\bibitem{ar1}
A.~Arkhangel'skii, \textit{Topological Function Spaces}, Kluwer Academic Publishers, Dordrecht, 1992.

\bibitem{bd}
J.~Baars and J.~de Groot, \textit{On topological and linear equivalence of certain function spaces}, CWI tract 86, Centre for for Mathematics and Computer Science, Amsterdam, 1992.

\bibitem{b1}
J.~Baars, \textit{Function spaces on first countable paracompact spaces}, Bull. Pol. Acad. Sci. Math. \textbf{42} (1994), 29--35.

\bibitem{b2}
J.~Baars, \textit{On the $l_p^{\ast}$-equivalence of metric spaces}, Topology Appl. \textbf{298} (2021), 107729.

\bibitem{b3}	
J.~Baars, \textit{Linear equivalence of scattered metric spaces}, Canad. Math. Bull. \textbf{64} (2023), 1354--1367.

\bibitem{Pelant} J.~Baars, J.~de Groot and J. Pelant,
\textit{Function spaces of completely metrizable spaces}, Trans. Amer. Math. Soc. \textbf{340} (1993), 871--883.

\bibitem{BMT}J. Baars, J. van Mill and V. V. Tkachuk,
\textit{Linear equivalence of (pseudo) compact spaces}, 
Quaestiones Math. \textbf{46:3} (2023), 513--518.

\bibitem{d}
J.~Dugundji, \textit{An extension of Tietze's theorem}, Pacific J. Math. \textbf{1} (1951), 353--367.

\bibitem{en}
R.~Engelking, \textit{Theory of dimensions, finite and infinite}, Sigma Series in Pure Mathematics, 10. Heldermann Verlag, 1995.

\bibitem{ev}
A.~Eysen and V.~Valov, \textit{On uniformly continuous surjections between function spaces}, https://arxiv.org/abs/2404.00542, (submitted for publication).


\bibitem{gf}
P.~Gartside and Z.~Feng, \textit{Spaces $l$-dominated by $\mathbb I$ or $\mathbb R$}, Topology Appl. \textbf{219} (2017), 1--8.

\bibitem{gkm}
R.~Gorak, M.~Krupski and W.~Marciszewski, \textit{On uniformly continuous maps between function spaces}, Fund. Math. \textbf{246} (2019), 257--274. 

\bibitem{gu}
S.~Gul'ko, \textit{On uniform homeomorphisms of spaces of continuous functions}, Trudy Mat. Inst. Steklov, \textbf{193} (1992), 82--88 (in Russian):
 English translation: Proc. Steklov Inst. Math. \textbf{193} (1992), 87--93. 


\bibitem{KL2} J. K\c akol and A. Leiderman,
 \textit{Basic properties of $X$ for which the space $C_p(X)$ is distinguished},
 Proc. Amer. Math. Soc., series B,  \textbf{8} (2021), 267--280.

\bibitem{KKL} J. K\c akol, O. Kurka and A. Leiderman,  
 \textit{Some classes of topological spaces extending the class of $\Delta$-spaces},
 Proc. Amer. Math. Soc. \textbf{152} (2024), 883--899.

\bibitem{KawL}
K.~Kawamura and A.~Leiderman, \textit{Linear continuous surjections of $C_p$-spaces over compacta},
 Topology Appl. \textbf{227} (2017), 135--145. 

\bibitem{k}
M.~Krupski, \textit{On $\kappa$-pseudocompactess and uniform homeomorphisms of function spaces}, Results Math. \textbf{78} (2023), 
no. 4, Paper No. 154, 11 pp.
 
\bibitem{llp}
A.~Leiderman, M.~Levin and V.~Pestov, \textit{On linear continuous open surjections of the spaces $C_p(X)$},
 Topology Appl. \textbf{81} (1997), 269--279.

\bibitem{lmp}
A.~Leiderman, S.~Morris and V.~Pestov, \textit{The free abelian topological group and the free locally convex space on the unit interval},
 J. London Math. Soc. \textbf{56} (1997), 529--538.

\bibitem{Levin}
M.~Levin, {\em A property of $C_{p}[0,1]$,} Trans. Amer. Math. Soc.
  \textbf{363} (2011), 2295--2304.
	
\bibitem{mar2}
W.~Marciszewski, \textit{On properties of metrizable space preserved by $t$-equivalence},
 Mathematika \textbf{47} (2000), 273--279.	

\bibitem{mar}
W. Marciszewski, \textit{Function spaces}, 
in M. Hu\v sek and J. van Mill eds., Recent progress in general topology II, Elseiver (2002), 345--369.

\bibitem{mp}
W.~Marciszewski and J.~Pelant, \textit{Absolute Borel sets and function spaces},
 Trans. Amer. Math. Soc. \textbf{349} (1997), 3585--3596.

\bibitem{Michael}
E. Michael, \textit{A note on completely metrizable spaces}, 
Proc. Amer. Math. Soc. \textbf{96} (1986), 513--522.

\bibitem{vanMill} J. van Mill,
\textit{The Infinite-Dimensional Topology of Function Spaces}, North-Holland Mathematical Library 64, North-Holland, Amsterdam, 2001.

\bibitem{p} V.~Pestov,
 \textit{The coincidence of the dimension $\dim$ of $l$-equivalent topological spaces}, Soviet Math. Dokl.
 \textbf{26} (1982), 380--383. 

\bibitem{Plewik} S. Plewik and M. Walczy\' nska,
 \textit{On metric $\sigma$-discrete spaces},
Algebra, logic and number theory, Banach Center Publ. 108,
 Institute of Math., Warsaw, 2016, 239--253.

\bibitem{Solomon} R. S. Solomon,
\textit{A scattered space that is not zero-dimensional},
Bull. London Math. Soc. \textbf{8} (1976), 239--240.

\bibitem{Stone} A. H. Stone,
\textit{Kernel constructions and Borel sets},
Trans. Amer. Math. Soc. \textbf{107} (1963), 58--70.

\bibitem{Tel} R. Telg\' arsky,
\textit{Total paracompactness and paracompact dispersed spaces},
Bull. Pol. Acad. Sci. Math. \textbf{16} (1968), 567--572.

\bibitem{tk}
V.~Tkachuk, \textit{$C_p$-theory problem book. Topological and function spaces. Problem Books in Mathematics}, Springer, Berlin, New York, 2011.

\bibitem{tk2}
V.~Tkachuk, \textit{$C_p$-theory problem book. Functional equivalencies. Problem Books in Mathematics}, Springer, Berlin, New York, 2016.

\end{thebibliography}
\end{document}